\documentclass{article}
\usepackage{amssymb}
\usepackage{amsmath}

\newcommand{\Eff}{{\cal E}\! f\! f}

\newtheorem{proposition}{Proposition}[section]
\newtheorem{lemma}[proposition]{Lemma}
\newtheorem{corollary}[proposition]{Corollary}
\newtheorem{definition}[proposition]{Definition}
\newtheorem{theorem}[proposition]{Theorem}
\newtheorem{exrcise}{Exercise}

\newcounter{opgaveteller}
\newcounter{lijst-teller}

\newcounter{boon}
\newcounter{boon2}
\newenvironment{proof}{\noindent {\bf Proof}. \nopagebreak }{\nopagebreak\hfill\rule{2mm}{3mm}}

\input xypic
\xyoption{all}
\UseComputerModernTips
\CompileMatrices
\definemorphism{dat}\dashed\tip\notip
\definemorphism{dub}\Solid\Tip\notip

\newenvironment{rlist}%
   {\begin{list}{\roman{lijst-teller}\/)\hfil}%
              {\labelwidth 2em%
               \leftmargin\labelwidth\advance\leftmargin by\labelsep%
               \usecounter{lijst-teller}}}%
   {\end{list}}

\newenvironment{r'list}%
   {\begin{list}{\roman{lijst-teller}\/)$'$\hfil}%
              {\labelwidth 2em%
               \leftmargin\labelwidth\advance\leftmargin by\labelsep%
               \usecounter{lijst-teller}}}%
   {\end{list}}

   {\end{list}}

    {\end{list}}

\title{Realizability with a Local Operator of A.M.\ Pitts}
\author{Jaap van Oosten\footnote{Department of Mathematics, Utrecht University, P.O.\ Box 80.010, 3508 TA Utrecht, The Netherlands, {\tt j.vanoosten@uu.nl}}}
\date{January 2, 2013}

\begin{document}
\maketitle
\begin{abstract}We study a notion of realizability with a local operator $\cal J$ which was first considered by A.M.\ Pitts in his thesis \cite{PittsAM:thet}. Using the Suslin-Kleene theorem, we show that the representable functions for this realizability are exactly the hyperarithmetical ($\Delta ^1_1$) functions.

We show that there is a realizability interpretation of nonstandard arithmetic, which, despite its classical character, lives in a very nonclassical universe, where the Uniformity Principle holds and K\"onig's Lemma fails. We conjecture that the local operator gives a useful indexing of the hyperarithmetical functions.\end{abstract}

\section*{Introduction}
This short note collects a few results from an analysis of a notion of realizability with a local operator first identified by A.M.\ Pitts. The notions `local operator' and the realizability to which it gives rise are defined in sections 1 and 2, respectively. We refer to this as `$\cal J$-realizability'.

This $\cal J$-realizability was studied in \cite{LeeS:basse} where it was established that all arithmetical functions are `$\cal J$-representable' (again, for a definition see section 2). Here we sharpen this result and characterize the $\cal J$-representable functions as exactly the hyperarithmetical ($\Delta ^1_1$) functions.

We show that there is a $\cal J$-realizability interpretation of nonstandard arithmetic. This is in sharp contrast to ordinary Kleene realizability, where there cannot even exist a nonstandard model of intuitionistic $I\Sigma _1$: see \cite{McCartyD:varth,BergB:aricat}.

Despite these `classical' features of $\cal J$-realizability, it forms part of a very non-classical universe, in which for example the Uniformity Principle holds and K\"onig's Lemma fails.

We conjecture that Pitts' local operator gives a neat indexing of the hyperarithmetical functions, which could be fruitful in developing `recursion theory with hyperrarithmetical functions' (a topic touched upon in chapter 16 of the classic \cite{RogersH:therfe}).

The paper starts out as concretely as possible, in an effort to be accessible to any reader who is familiar with realizability and recursion theory. More general and conceptual, topos-theoretic comments are therefore relegated to a final section, which can be skipped without detriment to the reader's understanding of the technical material presented before.
\section{Notation and Preliminaries}
We assume a recursive coding of finite sequences; the code of a sequence $\sigma = (a_0,\ldots ,a_{n-1})$ is written $\langle a_0,\ldots ,a_{n-1}\rangle$; we have a recursive function {\sf lh} giving the length of a coded sequence, and recursive projections $(\cdot )_i$, such that the following equations hold:
$$\begin{array}{ll}(\langle a_0,\ldots ,a_{n-1}\rangle )_i = a_i & 0\leq i<n \\
\langle (s)_0,\ldots ,(s)_{{\sf lh}(s)-1}\rangle = s & \end{array}$$

\noindent For subsets $A,B$ of $\mathbb{N}$ we write $A\to B$ for the set of indices of partial recursive functions which map every element of $A$ to some element in $B$ (and in particular, are defined on every element of $A$). We write $A\wedge B$ for the set $\{\langle a,b\rangle\, |\, a\in A,b\in B\}$.

The partial recursive function with index $e$ is denoted $\phi _e$. We employ $\lambda$-notation: the expression $\lambda x.t$ denotes a standard index (obtained by the S-m-n theorem) for the (partial) function $x\mapsto t$.

\begin{definition}\label{mondef}\em A function ${\cal F}: {\cal P}(\mathbb{N})\to {\cal P}(\mathbb{N})$ is {\em (recursively) monotone\/} if the set
$$\bigcap_{A,B\subseteq\mathbb{N}}(A\to B)\to ({\cal F}A\to {\cal F}B)$$
is nonempty.

The set of monotone functions is preordered as follows: we write ${\cal F}\leq {\cal G}$ if the set $\bigcap_{A\subseteq\mathbb{N}}{\cal F}A\to {\cal G}A$ is nonempty.\end{definition}
\begin{definition}\label{locdef}\em A function ${\cal J}: {\cal P}(\mathbb{N})\to {\cal P}(\mathbb{N})$ is a {\em local operator\/} if the following sets are nonempty:
$$\begin{array}{lll} E_1({\cal J}) & = & \bigcap_{A,B\subseteq\mathbb{N}}(A\to B)\to ({\cal J}A\to {\cal J}B)  \\
E_2({\cal J}) & = & \bigcap_{A\subseteq\mathbb{N}}A\to {\cal J}A \\
E_3({\cal J}) & = & \bigcap_{A\subseteq\mathbb{N}}{\cal J}{\cal J}A\to {\cal J}A\end{array}$$\end{definition}
So, every local operator is a monotone function. Examples of local operators are: the function which maps every set to $\mathbb{N}$ (the {\em trivial\/} local operator) and the function which maps $\emptyset$ to $\emptyset$ and every nonempty set to $\mathbb{N}$ (the $\neg\neg$-operator).

It is left to the reader to verify that from elements of $E_1({\cal J}), E_2({\cal J}),E_3({\cal J})$ we can recursively obtain an element of
$$E_4({\cal J})\; =\; \bigcap_{A,B\subseteq\mathbb{N}}{\cal J}A\wedge {\cal J}B\to {\cal J}(A\wedge B)$$

The following theorem was proved in \cite{PittsAM:thet} and \cite{HylandJ:efft}.
\begin{theorem}[Hyland-Pitts]\label{hylandpitts} For any recursively monotone function $\cal F$ there is a least (w.r.t.\ the preorder on monotone functions) local operator $L({\cal F})$ with the property that ${\cal F}\leq L({\cal F})$.

An explicit formula for $L({\cal F})$ is
$$L({\cal F})A\; =\; \bigcap\{B\subseteq\mathbb{N}\, |\,\{ 0\}\wedge A\subseteq B\text{ and }\{ 1\}\wedge {\cal F}B\subseteq B\}$$\end{theorem}
For more on local operators, the reader is referred to \cite{LeeS:basse}.

\noindent In this paper, we shall deal with only one monotone function $\cal F$ and its associated local operator $L({\cal F})$. This function was defined by A.M.\ Pitts in \cite{PittsAM:thet}:
$${\cal F}A\; =\; \bigcup_{n\in\mathbb{N}}({\uparrow}n\to A)$$
where ${\uparrow}n$ is short for $\{ m\in\mathbb{N}\, |\, n\leq m\}$. Henceforth we write $\cal J$ for this $L({\cal F})$.

Pitts proved the following facts:
\begin{lemma}\label{pittslemma}~~~~~~~~~~~~~~~~~~~~~~~~\begin{rlist}
\item ${\cal J}\emptyset = \emptyset$
\item ${\cal J}\{ 0\} \cap {\cal J}\{ 1\} =\emptyset$
\item ${\cal J}$ preserves inclusions.\end{rlist}\end{lemma}
From items i) and ii) it follows that $\cal J$ is not the $\neg\neg$-operator.

\noindent We reserve the letters {\sf a}, {\sf b}, {\sf c}, {\sf d}, {\sf e} for chosen elements of the following sets:
$$\begin{array}{lll} {\sf a} & \in & \bigcap_{A\subseteq\mathbb{N}}A\to {\cal J}A \\
{\sf b} & \in & \bigcap_{A,B\subseteq\mathbb{N}}(A\to B)\to ({\cal J}A\to {\cal J}B) \\
{\sf c} & \in & \bigcap_{A\subseteq\mathbb{N}}{\cal F}A\to {\cal J}A \\
{\sf d} & \in & \bigcap_{A\subseteq\mathbb{N}}{\cal J}{\cal J}A\to {\cal J}A \\
{\sf e} & \in & \bigcap_{A,B\subseteq\mathbb{N}}{\cal J}A\wedge {\cal J}B\to {\cal J}(A\wedge B)\end{array}$$

\noindent The following lemma was proved in \cite{LeeS:basse}.
\begin{lemma}\label{rijtjeslemma} For any total recursive function $F$ there is a partial recursive function $G$ (an index for which can be obtained recursively in an index for $F$), such that for every coded sequence $s=\langle a_0,\ldots ,a_{n-1}\rangle$ and every $n$-tuple $x_0,\ldots ,x_{n-1}$ such that $x_0\in {\cal J}\{ a_0\},\ldots ,x_{n-1}\in {\cal J}\{ a_{n-1}\}$, we have
$$G(\langle x_0,\ldots ,x_{n-1}\rangle )\; \in\; {\cal J}\{ F(s)\}$$\end{lemma}
The following corollary is easy, and just stated for easy reference:
\begin{corollary}\label{rijtjescor} There are partial recursive functions $G$ and $H$ such that for $x_0\in {\cal J}\{ a_0\},\ldots ,x_{n-1}\in {\cal J}\{ a_{n-1}\}$ we have
$$\begin{array}{llll} G(\langle x_0,\ldots ,x_{n-1}\rangle ) & \in & {\cal J}\{ 0\} & \text{if for some $i<n$, $a_i=0$} \\
G(\langle x_0,\ldots ,x_{n-1}\rangle ) & \in & {\cal J}\{ 1\} & \text{otherwise}\\
H(\langle x_0,\ldots ,x_{n-1}\rangle ) & \in & {\cal J}\{ i\} & \text{if $i<n$ is least such that $a_i=0$}\\
H(\langle x_0,\ldots ,x_{n-1}\rangle ) & \in & {\cal J}\{ n\} & \text{if there is no such $i<n$}
\end{array}$$
\end{corollary}

\section{$\cal J$-Assemblies and $\cal J$-realizability}
The category of $\cal J$-{\em assemblies} has as objects pairs $(X,E)$ where $X$ is a set and $E$ a function which assigns to every $x\in X$ a nonempty set $E(x)\subseteq\mathbb{N}$. A {\em morphism\/} of $\cal J$-assemblies $(X,E)\to (Y,F)$ is a function $f:X\to Y$ such that the set
$$\bigcap_{x\in X}E(x)\to {\cal J}F(f(x))$$
is nonempty; any element of this set is said to {\em track\/} the function $f$.

Morphisms can be composed: given $f:(X,E)\to (Y,F)$ and $g:(Y,F)\to (Z,G)$, tracked by $n$ and $m$ respectively, then
$$\lambda v.\phi _{\sf d}(\phi _{\phi _{\sf b}(m)}(\phi _{\sf n}(v)))$$
tracks $gf$, as is easy to check.

The category of $\cal J$-assemblies is cartesian closed: the product of $\cal J$-assemblies $(X,E)$ and $(Y,F)$ can be given as $(X\times Y,G)$ where $G(x,y)=E(x)\wedge F(y)$. The exponent $(Y,F)^{(X,E)}$ has as underlying set the set of morphisms from $(X,E)$ to $(Y,F)$; and assigns to such a morphism the set of its trackings. Moreover, the category has a {\em natural numbers object}: the object $N\; =\; (\mathbb{N},\{\cdot\} )$.

For a $\cal J$-assembly $(X,E)$, a {\em subobject\/} is given by a function $R:X\to {\cal P}(\mathbb{N})$ such that the set $\bigcap_{x\in X}R(x)\to{\cal J}E(x)$ is nonempty; this data determines a $\cal J$-assembly $(X',R)$, where $X'=\{ x\in X\, |\, R(x)\neq\emptyset\}$, and a monomorphism $(X',R)\to (X,E)$.

${\cal J}$-assemblies can be structures for a first-order language: suppose $(X,E)$ is a $\cal J$-assembly; suppose $n$-ary function symbols $f$ of the language are interpreted as morphisms $[f]:(X,E)^n\to (X,E)$, and $n$-ary relation symbols $R$ by subobjects $[R]$ of $(X,E)^n$ (thought of as maps $[R]:X^n\to {\cal P}(\mathbb{N})$). 

We have a notion of truth given by $\cal J$-{\em realizability}. We define, for a formula $\varphi (v_1,\ldots ,v_n)$ of the language and elements $x_1,\ldots ,x_n$ of $X$, what it means that a natural number $e$ $\cal J$-{\em realizes} $\phi (x_1,\ldots ,x_n)$:\begin{itemize}
\item[] $e$ $\cal J$-realizes $t=s(\vec{x})$ iff $e\in {\cal J}E(x_1)\wedge\cdots\wedge {\cal J}E(x_n)$ and $[t](\vec{x})=[s](\vec{x})$
\item[] $e$ $\cal J$-realizes $R(\vec{x})$ iff $e\in [R](\vec{x})$
\item[] $e$ $\cal J$-realizes $(\phi\wedge\psi )(\vec{x})$ iff $(e)_0$ $\cal J$-realizes $\phi (\vec{x})$ and $(e)_1$ $\cal J$-realizes $\psi (\vec{x})$
\item[] $e$ $\cal J$-realizes $(\phi\vee\psi )(\vec{x})$ iff either $(e)_0=0$ and $(e)_1$ $\cal J$-realizes $\phi (\vec{x})$, or $(e)_0\neq 0$ and $(e)_1$ $\cal J$-realizes $\psi (\vec{x})$
\item[] $e$ $\cal J$-realizes $(\phi\to\psi )(\vec{x})$ iff $(e)_0\in {\cal J}E(x_1)\wedge\cdots\wedge {\cal J}E(x_n)$ and for all $m$ such that $m$ $\cal J$-realizes $\phi (\vec{x})$, $\phi _{(e)_1}(m)$ is defined and is an element of ${\cal J}\{ k\, |\, k\text{ $\cal J$-realizes }\psi (\vec{x})\}$
\item[] $e$ $\cal J$-realizes $\exists x\phi (\vec{x})$ iff for some $a\in X$, $(e)_0\in {\cal J}E(a)$ and $(e)_1$ $\cal J$-realizes $\phi (a,\vec{x})$
\item[] $e$ $\cal J$-realizes $\forall x\phi (\vec{x})$ iff $(e)_0\in {\cal J}E(x_1)\wedge\cdots\wedge {\cal J}E(x_n)$ and for all $y\in X$ and all $k\in E(y)$, $\phi _{(e)_1}(k)$ is defined and an element of ${\cal J}\{ m\, |\, m\text{ $\cal J$-realizes }\phi (y,\vec{x})\}$
\end{itemize}
In particular, this can be applied to the natural numbers object $N$ and the language of arithmetic. It was proved in \cite{LeeS:basse} that an arithmetical sentence is true under $\cal J$-realizability (i.e., has a $\cal J$-realizer) precisely if it is classically true.

This theorem was based on considering $\cal J$-{\em decidable\/} subsets of $\mathbb{N}$, and $\cal J$-{\em representable\/} functions $\mathbb{N}\to\mathbb{N}$.
\begin{definition}\label{deciddef}\em A subset $A\subseteq\mathbb{N}$ is called $\cal J$-{\em decidable\/} if there is a total recursive function $F$ such that $F(n)\in {\cal J}\{ 0\}$ if $n\in A$, and $F(n)\in {\cal J}\{ 1\}$ if $n\not\in A$.

A function $f:\mathbb{N}\to\mathbb{N}$ is $\cal J$-{\em representable\/} if there is a total recursive function $F$ such that for all $n\in\mathbb{N}$, $F(n)\in {\cal J}\{ f(n)\}$.\end{definition}

In \cite{LeeS:basse} it was shown that every arithmetical subset of $\mathbb{N}$ is $\cal J$-decidable; the following theorem sharpens this result. 

Recall that a subset $A$ of $\mathbb{N}$ is $\Pi ^1_1$ if it can be defined in the language of second-order arithmetic by a formula $A=\{ x\, |\, \forall X\psi (X,x)\}$ where $\forall X$ is the only second-order quantifier in $\forall X\psi (X,x)$. A set is $\Sigma ^1_1$ if its complement is $\Pi ^1 _1$; and a set is {\em hyperarithmetical\/} or $\Delta ^1 _1$, if it is both $\Pi ^1 _1$ and $\Sigma ^1 _1$. A function $f:\mathbb{N}\to\mathbb{N}$ is hyperarithmetical if
$${\rm graph}(f)\; =\; \{\langle n,f(n)\rangle\, |\, n\in\mathbb{N}\}$$
is a hyperarithmetical set.
\begin{theorem}\label{hyperarth} The $\cal J$-decidable sets are precisely the hyperarithmetical sets, and the $\cal J$-representable functions are precisely the hyperarithmetical functions.\end{theorem}
\begin{proof} Recall that ${\cal F}A=\bigcup_{n\in\mathbb{N}}{\uparrow}n\to A$, so ${\cal F}A$ is defined by an arithmetical formula in $A$. By the explicit formula for ${\cal J}=L({\cal F})$ given in theorem~\ref{hylandpitts}, we see that ${\cal J}A$ is defined by a formula
$${\cal J}A\; =\; \{x\, |\,\forall B(\psi (A,B)\to x\in B)\}$$
with $\psi (A,B)$ arithmetical in $A$. It follows that if $A$ is arithmetical, then ${\cal J}A$ is a $\Pi ^1_1$-set. In particular, ${\cal J}\{ 0\}$ is $\Pi ^1_1$. Hence, if $A\subseteq\mathbb{N}$ is $\cal J$-decided by the recursive function $F$ in the sense of definition~\ref{deciddef}, then $A=F^{-1}({\cal J}\{ 0\} )$, so also $\Pi ^1_1$. Since the complement of $A$ is $F^{-1}({\cal J}\{ 1\} )$ hence also $\Pi ^1_1$, it follows that $A$ is hyperarithmetical.

For the converse, in order to show that every hyperarithmetical set is $\cal J$-decidable, we consider the set
$$C\; =\; \{ e\, |\,\phi _e\text{ is total and for all $n\in\mathbb{N}$, }\phi _e(n)\in {\cal J}\{ 0\}\cup {\cal J}\{ 1\}\}$$
and the $C$-indexed collection of subsets of $\mathbb{N}$:
$$C_e\; =\; (\phi _e)^{-1}({\cal J}\{ 0\} )$$
Recall from lemma~\ref{pittslemma} that ${\cal J}\{ 0\}\cap {\cal J}\{ 1\} =\emptyset$, so the collection $\{ C_e\, |\, e\in C\}$ consists precisely of the $\cal J$-decidable sets. We need to show that it contains all $\Delta ^1_1$-sets.

This, in fact, is a straightforward application of the Suslin-Kleene Theorem (see \cite{OdifreddiP:clart,MoschovakisY:eleias}). We have to check that our collection $\{ C_e\, |\, e\in C\}$ is a so-called {\em SK-class\/} (\cite{OdifreddiP:clart}) or an {\em effective $\sigma$-ring\/} (\cite{MoschovakisY:eleias}). This means that we must exhibit partial recursive functions $\tau _1$, $\tau _2$ and $\sigma$ for which the following hold:\begin{rlist}\item For all $n$, $\tau _1(n)$ is defined and $C_{\tau _1(n)}=\{ n\}$
\item For all $e\in C$, $\tau _2(e)$ is defined and $C_{\tau _2(e)}=\mathbb{N}-C_e$
\item For every $e$ such that $\phi _e$ is total and $\phi _e$ takes values in $C$, $\sigma (e)$ is defined and
$$C_{\sigma (e)}=\bigcup_{n\in\mathbb{N}}C_{\phi _e(n)}$$\end{rlist}
The Suslin-Kleene theorem asserts that there is an indexing $\{ G_x\, |\, x\in G\}$ of the $\Delta ^1_1$-sets, which is the minimal SK-class (in an effective sense, which need not concern us here). So if we have proved i)--iii), it follows that $\{ C_e\, |\, e\in C\}$ contains all the $\Delta ^1_1$-sets.

For i) let $\chi _n(x)=\left\{\begin{array}{ll}0 & \text{if $x=n$}\\ 1 & \text{otherwise}\end{array}\right.$ and let $\tau _1(n)=\lambda x.\phi _{\sf a}(\chi _n(x))$

For ii) let $c$ be such that $\phi _c(0)=1$ and $\phi _c(1)=0$. Let $\tau _2(e)=\lambda x.\phi _{\phi _{\sf b}(c)}(\phi _e(x))$.

For iii)  let $G$ be a recursive function as in corollary~\ref{rijtjescor}. Now if $\phi _e$ is total and takes values in $C$, and $x\in\mathbb{N}$ is arbitrary, we have:\begin{itemize}
\item[] if $x\in\bigcup_{n\in\mathbb{N}}C_{\phi _e(n)}$ then
$$G(\langle \phi _{\phi _e(0)}(x),\ldots ,\phi _{\phi _e(n)}(x)\rangle )\in {\cal J}\{ 0\}$$
for $n$ large enough;
\item[] if $x\not\in\bigcup_{n\in\mathbb{N}}C_{\phi _e(n)}$ then
$$G(\langle \phi _{\phi _e(0)}(x),\ldots ,\phi _{\phi _e(n)}(x)\rangle )\in {\cal J}\{ 1\}$$
always.\end{itemize}
So if 
$$\chi (x)=\left\{\begin{array}{ll}0 & \text{if }x\in\bigcup_{n}C_{\phi _e(n)}\\ 1 & \text{else}\end{array}\right\}\;\;\text{and}\;\; \psi(e,x)\, =\, \lambda n.G(\langle \phi _{\phi _e(0)}(x),\ldots ,\phi _{\phi _e(n)}(x)\rangle )$$
then $\psi (e,x)\in {\cal F}{\cal J}\{\chi (x)\}$. So, let
$\sigma (e)\, =\, \lambda x.\phi _{\sf d}(\phi _{\sf c}(\psi (e,x)))$.
\medskip

\noindent For the statement about the $\cal J$-representable functions: clearly, if $f$ is $\cal J$-representable then ${\rm graph}(f)$ is a $\cal J$-decidable subset of $\mathbb{N}$, hence hyperarithmetical by the first part of the proof. Conversely, if ${\rm graph}(f)$ is $\cal J$-decidable we can find an index for a function which $\cal J$-represents $f$ by using the function $H$ from corollary~\ref{rijtjescor} in a way similar to what we have done in the first part, since $f(x)$ is the least $y$ such that $\langle x,y\rangle\in {\rm graph}(f)$.\end{proof}
\section{A $\cal J$-realizability interpretation of nonstandard arithmetic}
In \cite{SkolemT:niczme}, the first nonstandard model of Peano Arithmetic was constructed. Since the construction does not appear to be well-known and because elements of it are essential for what follows, we outline it here.

Let $\alpha _0,\alpha _1,\ldots$ be an enumeration of all arithmetical functions $\mathbb{N}\to\mathbb{N}$. We construct a strictly increasing function $\psi$ such that for all $i,j\in\mathbb{N}$ we have one of three possibilities: $\alpha _i\psi (n)<\alpha _j\psi (n)$ for almost all $n$, or $\alpha _i\psi (n)=\alpha _j\psi (n)$ for almost all $n$, or $\alpha _i\psi (n)>\alpha _j\psi (n)$ for almost all $n$.

In order to achieve this, one constructs a sequence $A_0\supset A_1\supset\cdots$ of infinite sets; each $A_k$ must have the property that for all $i,j\leq k$, $\alpha _i<\alpha _j$ on $A_k$ or $\alpha _i=\alpha _j$ on $A_k$ or $\alpha _i>\alpha _j$ on $A_k$. This is done as follows: let $A_0=\mathbb{N}$. Suppose inductively, that $A_k$ has been constructed and has the required property. Suppose that the restrictions of $\alpha _0,\ldots ,\alpha _k$ to $A_k$ are ordered as $\beta _1<\cdots <\beta _l$. Now $A_k$ can be written as a finite union
$$\begin{array}{lll}A_k & = & \{ x\in A_k\, |\,\alpha _{k+1}(x)<\beta _1(x)\} \\
& & \cup\{ x\in A_k\, |\,\alpha _{k+1}(x)=\beta _1(x)\} \\
& & \cup\{ x\in A_k\, |\,\beta _1(x)<\alpha _{k+1}(x)<\beta _2(x)\} \\
& & \cup\cdots \\
& & \cup\{ x\in A_k\, |\,\beta _l(x)<\alpha _{k+1}(x)\}\end{array}$$
Let $A_{k+1}$ be the first set in this list which is infinite. This completes the construction of the sequence $A_0\supset A_1\supset\cdots$.

Finally let $\psi$ be defined by: $\psi (0)=0$ and $\psi (k+1)$ is the least element of $A_{k+1}$ which is $>\psi (k)$.

The underlying set of Skolem's model is the set $\cal N$ of equivalence classes of arithmetical functions, where two such functions $\alpha$ and $\beta$ are equivalent if $\alpha\psi (n)=\beta\psi  (n)$ for $n$ large enough. We have an embedding $\iota :\mathbb{N}\to {\cal N}$ which sends $n$ to (the equivalence class of) the constant function with value $n$. We can extend an arithmetical function $\alpha :\mathbb{N}\to\mathbb{N}$ to $\cal N$ by putting $\alpha ([\beta ])=[\alpha\beta ]$; this is well-defined on equivalence classes, so $\cal N$ is a structure for the language of arithmetic; and $\iota$ is an elementary embedding since we can prove for any formula $\varphi (v_1,\ldots ,v_n)$ in the language of arithmetic and any $n$-tuple $[\beta _1],\ldots ,[\beta _n]$ of elements of $\cal N$, that ${\cal N}\models\varphi ([\beta _1],\ldots ,[\beta (n)])$ if and only if $\mathbb{N}\models\varphi (\beta _1\psi (k),\ldots ,\beta _n\psi (k))$ for almost all $k$.
\medskip

\noindent Now it is not hard to see that the whole construction, which needs an enumeration of all arithmetical functions and checking whether or not an arithmetical set is infinite, can be done recursively in a truth function for arithmetic, which is hyperarithmetical (see, e.g., \cite{RogersH:therfe}, 16-XI). Therefore, the function $\psi$ can be assumed to be $\cal J$-representable.

We can now endow the set $\cal N$ with the structure of a $\cal J$-assembly, by putting
$$E([\alpha ])\; =\; \{ e\, |\, \text{for some $\beta\in [\alpha ]$, $e$ represents $\beta\psi$}\} $$
For any arithmetical $\beta$, the map $[\alpha ]\to [\beta\alpha ]$ is well-defined and tracked, so the $\cal J$-assembly $\cal N$ is also a structure for the language of arithmetic. And again, we have an embedding $i:N\to {\cal N}$ of $\cal J$-assemblies, which is just $\iota$ on the level of sets.

By a straightforward application of the proof method in \cite{LeeS:basse} for the theorem that the $\cal J$-realizable sentences of arithmetic are exactly the classically true ones, one now obtains the following therem.
\begin{theorem}\label{Nelextthm} The map $i$ is an elementary embedding. For a formula $\varphi (v_1,\ldots ,v_n)$ and numbers $a_1,\ldots ,a_n$ the following four assertions are equivalent:\begin{rlist}
\item $\varphi (a_1,\ldots ,a_n)$ is true in the classical model $\mathbb{N}$
\item $\varphi (a_1,\ldots ,a_n)$ has a $\cal J$-realizer (in the sense of the assembly $N$)
\item $\varphi (i(a_1),\ldots ,i(a_n))$ has a $\cal J$-realizer (in the sense of the assembly $\cal N$)
\item $\varphi (i(a_1),\ldots ,i(a_n))$ is true in the classical model $\cal N$\end{rlist}
Moreover, the equivalence ii)$\Leftrightarrow$iii) is effective in realizers.

If $\alpha _1,\ldots ,\alpha _k$ are arithmetical functions then the following are equivalent:\begin{rlist}
\item $\varphi ([\alpha _1],\ldots ,[\alpha _n])$ is true in the classical model $\cal N$
\item $\varphi ([\alpha _1],\ldots ,[\alpha _n])$ has a $\cal J$-realizer
\item $\varphi (\alpha _1\psi (k),\ldots ,\alpha _n\psi (k))$ is true in $\mathbb{N}$ for almost all $k$\end{rlist}\end{theorem}
The model $\cal N$ is in fact very classical: let St (the subobject of standard numbers) denote the image of $i:N\to {\cal N}$. Since the condition `$\alpha$ is bounded' is arithmetical in $\alpha$, we have:
\begin{proposition}\label{stvnonst} The statement $\forall y({\rm St}(y)\vee\neg {\rm St}(y))$ has a $\cal J$-realizer.\end{proposition}
Nevertheless, the universe of $\cal J$-assemblies also has non-classical features. Just like in the category of ordinary assemblies, K\"onig's Lemma fails, and Cantor space and Baire space are isomorphic:
\begin{proposition}\label{noKL} In the category of $\cal J$-assemblies, the objects $2^N$ and $N^N$ are isomorphic. Hence, K\"onig's Lemma fails: there is a continuous but unbounded function $2^N\to N$.\end{proposition}
\begin{proof} This follows from the result (see \cite{RogersH:therfe}, Corollary 16-XLI(b)) that, analogous to the ordinary Kleene tree, there is a recursive, finitely-branching, infinite tree which has no infinite hyperarithmetical branch. The stated isomorphism now follows in a way similar to \cite{OostenJ:reaics}, 3.2.26 (see also \cite{HylandJ:efft}, 13.1--4).\end{proof}
\section{General comments and further work}
Just as ordinary Kleene realizability is the standard notion of truth in an elementary topos, the {\em effective topos\/} $\Eff$ of J.M.E.\ Hyland (\cite{HylandJ:efft}), $\cal J$-realizability is the standard notion of truth in a topos, a {\em subtopos\/} of the effective topos. Let us denote this subtopos by ${\Eff}_{\cal J}$.
The topos ${\Eff}_{\cal J}$ shares some features with $\Eff$: it is the free exact completion over the regular category of $\cal J$-assemblies. Every object is covered by a $\cal J$-assembly. The subobject classifier $\Omega$ is the object $({\cal P}(\mathbb{N}),=)$ where $[A=B]$ is the set $(A\to {\cal J}B)\wedge (B\to {\cal J}A)$. It is immediate that $\langle {\sf a},{\sf a}\rangle$ is an element of $[A=A]$ for all $A$, and this implies that the {\em Uniformity Principle\/} holds:
\begin{proposition}\label{upineffj}For any object $X$ which is a subquotient of $N$, the natural map $X\to X^{\Omega}$ is an isomorphism. In particular, this holds for the objects $N$, $N^N$, $\cal N$ and ${\cal N}^{\cal N}$.\end{proposition}
Analogies between ${\Eff}_{\cal J}$ and $\Eff$ can also be drawn on the basis of an analysis of the (partial) hyperarithmetical functions and the indexing to which the local operator $\cal J$ gives rise: write $F=\psi _e$ if for every $n$: $n\in {\rm dom}(F)$ if and only if $\phi _e(n)\in {\cal J}\{m\}$ for some (necessarily unique) $m$, and $\phi _e(n)\in {\cal J}\{ F(n)\}$ if $n\in {\rm dom}(F)$.

One sees that ${\rm dom}(\psi _e)$ is a $\Pi ^1_1$-set; this is in accordance with the philosophy of `recursion theory with hyperarithmetical functions', that if the latter are analogous to recursive functions, the analogues of r.e.\ sets are the $\Pi ^1 _1$-sets (\cite{RogersH:therfe}, p.402).

We conjecture that (a subcollection of) the $\Pi ^1_1$-sets form a {\em dominance\/} in ${\Eff}_{\cal J}$ (\cite{RosoliniG:cet,OostenJ:reaics}) and that there is a model of Synthetic Domain Theory in this topos.

Finally, let us remark that the nonstandard model given here, should be compared with the model defined in \cite{MoerdijkI:minsta}, section 3. In both cases it is a model in a sheaf topos over $\Eff$, and there is an obvious similarity between (the monotone function generating) Pitts' local operator and the Fr\'echet filter in $\Eff$. But
our proposition~\ref{stvnonst} contrasts which what Moerdijk claims to hold in his model (proposition 3.1 in \cite{MoerdijkI:minsta}).

\begin{small}
\bibliographystyle{plain}

\end{small}

\end{document}